\documentclass[11pt]{article}

\usepackage{amsmath}
\usepackage{amssymb}
\usepackage{amsthm}
\usepackage{tikz}
\usepackage{mathtools}
\usepackage{arydshln}

\usepackage{multirow}
\usepackage{enumitem}

\newtheorem{proposition}{Proposition}[section]

\newtheorem{corollary}{Corollary}[section]
\theoremstyle{definition}

\newtheorem{example}{Example}[section]

\theoremstyle{plain}

\newcommand{\E}{\mathmakebox[\widthof{3}]{\underset{\triangle}{3}}}

\newcommand{\nxt}{\operatorname{next}}
\newcommand{\mex}{\operatorname{mex}}

\begin{document}

\title{Infinite Subtraction Games with Periodic Outcomes and Aperiodic SG Values}

\author{Kai Liang\thanks{School of Mathematics and Statistics, Xidian University, 266 Xinglong Section, Xifeng Road, 710126 Xi'an, Shaanxi, China \\ 24071213197@stu.xidian.edu.cn}}
\date{\today}
\maketitle

\noindent\textbf{Revised Note (August 2026):}
A previous version of this preprint, posted on 18 Aug 2026, claimed to solve an open problem as the title states.
The author have since become aware that this problem, along with numerous related open questions in subtraction games and similar games, has been extensively studied and resolved in the work of Lomi{\v{c}}~\cite{Lomic2019}, which was inadvertently overlooked during our initial literature review.
Consequently, our contribution does not constitute a novel solution, but rather provides an additional example illustrating the framework established by Lomi{\v{c}}.
In light of this, we have decided that this preprint will not be submitted for journal publication.
We encourage readers to consult Lomi{\v{c}}'s paper for a comprehensive treatment of the subject.

\begin{abstract}
     We find some significant special cases of subtraction games with infinite subtraction sets, whose outcome sequences are periodic, but whose SG value sequences are bounded and aperiodic.
\end{abstract}

\section{Definitions and Notation}

\textit{Subtraction games}~\cite{SubGame1966} are among the simplest impartial games in combinatorial game theory.
The rules are as follows: given a nonempty set $S$ of positive integers (called the \textit{subtraction set}, typically assumed to be finite) and an initial position consisting of a heap of $u$ tokens (or beans), two players alternately remove $s \in S$ tokens from the heap.
The player who is first unable to move loses, and the other player wins.

In this paper, we use the natural number $u \in \mathbb{N}$ to denote the position with $u$ tokens in a subtraction game.
We shall also write $\mathcal{H} \coloneq\mathbb{N}$ to denote the set of all positions, in order to distinguish it from natural numbers in the usual sense.

For any integer set $N\subseteq\mathbb{Z}$, and $m\in\mathbb{Z}$, we will also use the notation like
\begin{align*}
    N^{>m}     &\coloneq \{n \mid n\in N,\ n> m\}; \\
    N^{\geq m} &\coloneq \{n \mid n\in N,\ n\geq m\}; \\
    N\pm m     &\coloneq \{n\pm m \mid n\in N\}; \\
    mN         &\coloneq \{mn \mid n\in N\}.
\end{align*}

Let $\nxt(u)$ denote the set of all successors of position $u$.
For subtraction games, using our notation, it can be simply given by
\[
    \nxt(u) \coloneq (u-S)^{\geq 0}.
\]

We use $\mathcal{P}$ and $\mathcal{N}$ to denote the sets of positions that are previous-player wins and current-player wins, respectively.
Since every subtraction game starting from any position terminates after finitely many moves, the outcome of a position can be defined recursively by
\begin{align*}
    u \in \mathcal{P} &\iff \nexists u'\in \nxt(u),\ u'\in\mathcal{P}; \\
    u \in \mathcal{N} &\iff \exists  u'\in \nxt(u),\ u'\in\mathcal{P}.
\end{align*}
This definition already implies that all terminal positions are necessarily $\mathcal{P}$-positions.

We use $o(u)$ to denote the \textit{outcome} of position $u$, adopting the same notation as the sets above:
\[
    o(u) \coloneq
    \begin{cases}
        \mathcal{P}, & u\in\mathcal{P};\\
        \mathcal{N}, & u\in\mathcal{N}.
    \end{cases}
\]

We use $\mathcal{G}(u)$ to denote the \textit{SG value} (or \textit{nim value})~\cite{Sprague1935, Grundy1939} of position $u$, defined as
\[
    \mathcal{G}(u) \coloneq \mex(\{ \mathcal{G}(u') \mid u'\in\nxt(u)\}),
\]
where the $\mex$ function
\[
    \mex(N) \coloneq \min (\mathbb{N} \setminus N)
\]
denotes the smallest nonnegative integer not appearing in a given set $N$ of nonnegative integers.

By the Sprague--Grundy theorem,
\[
    o(u) = \mathcal{P} \iff \mathcal{G}(u) = 0.
\]

In this paper, periodicity is by default referred to as \textit{ultimately periodicity},
namely, for a sequence \(a(n), n\in\mathbb{N}\), there exist a preperiod \(p\in\mathbb{N}\) and a period \(t\in\mathbb{N}^{\geq 1}\) (both taken to be minimal) such that for all \(i \geq p\),
\[
    a(i+t)=a(i).
\]

\section{Problem Background}

In the standard subtraction game, the subtraction set $S$ is finite.
In this case, the sequences of outcomes and SG values can always be given by a recurrence of the form.
For a position $u\in\mathcal{H}$ with $u\geq \max(S)$, we have
\begin{align*}
    o(u) &\coloneq
    \begin{cases}
        \mathcal{N}, & \mathcal{P}\in \{o(u-s) \mid s\in S\}; \\
        \mathcal{P}, & \mathcal{P}\notin \{o(u-s) \mid s\in S\}.
    \end{cases} \\
    \mathcal{G}(u) &\coloneq \mex(\{ \mathcal{G}(u-s) \mid s\in S \}).
\end{align*}

From this, it is not difficult to see that both sequences are always ultimately periodic.
However, if the subtraction set is infinite (in which case we call the game an \textit{infinite subtraction game}), the situation becomes much more complicated.
One may then ask: does there exist a game whose outcome sequence is periodic but whose SG sequence is not?
This question appears as an open problem in~\cite{GONC6}:

``
    An old, voiced but not written, question asks for an infinite subtraction set where the associated outcome-sequence is periodic but the nim-sequence is aperiodic.
    These are easy to find in octal games.
''

In this paper, we construct some examples to demonstrate that such games do indeed exist.

\section{An Example}

We first present one of the simplest example among those we have found.
Its subtraction set is
\begin{example}\label{S_e1}
    \begin{align*}
        S &= \{5\} \cup \{3 \times 2^p - 2 \mid p\in\mathbb{N}\} \\
          &= \{5, 1, 4, 10, 22, 46, 94, 190,382, \ldots\}.
    \end{align*}
\end{example}

A simple program can compute the first 120 SG values as:
\begin{center}
    \begin{tabular}{|c|c|}
        \hline
        $u$ & $\mathcal{G}(u)$ \\
        \hline
         & \\[-8pt]
        0+  & $0~ 1~ 0~ 1~ 2~ 3~ 2~ 3~ 0~ 1~ 4~ 0~ 1~\E~ 0~ 1~\E~ 0~ 1~ 2~ 0~ 1~ 2~ 0~ 1~ 2~ 0~ 1~ 2~ 0~$ \\[8pt]
        30+ & $1~\E~ 0~ 1~ 2~ 0~ 1~ 2~ 0~ 1~ 2~ 0~ 1~\E~ 0~ 1~ 2~ 0~ 1~ 2~ 0~ 1~ 2~ 0~ 1~ 2~ 0~ 1~ 2~ 0~$ \\[8pt]
        60+ & $1~ 2~ 0~ 1~\E~ 0~ 1~ 2~ 0~ 1~ 2~ 0~ 1~ 2~ 0~ 1~ 2~ 0~ 1~ 2~ 0~ 1~ 2~ 0~ 1~ 2~ 0~ 1~ 2~ 0~$ \\[8pt]
        90+ & $1~ 2~ 0~ 1~\E~ 0~ 1~ 2~ 0~ 1~ 2~ 0~ 1~ 2~ 0~ 1~ 2~ 0~ 1~ 2~ 0~ 1~ 2~ 0~ 1~ 2~ 0~ 1~ 2~ 0~$ \\[8pt]
        $\ldots$ & $\ldots$ \\
        \hline
    \end{tabular}
\end{center}

It can be observed that the SG sequence is close to the periodic sequence
\[
    0~ 1~ 0~ 1~ 2~ 3~ 2~ 3~ 0~ 1~ 4~ (0~ 1~ 2)^{\omega}
\]
except that some of the 2's in the periodic part $(0~1~2)^{\omega}$ are replaced by 3's (marked as $\E$),
and the gaps between these 3's grow increasingly large, thereby destroying the periodicity.

Let $\mathcal{H}_g$ denote the set of positions with SG value $g$.
The set of positions with SG value $3$ in this game is
\[
    \mathcal{H}_3 = \{5, 7, 16, 28, 52, 100, 196, 388,\ldots\}
\]
and it is not hard to notice that precisely
\[
    \mathcal{H}_3 = \{5, 7\} \cup (S^{\geq10}+6)
\]

\section{Conclusion and Proof}

The following result gives a special class of infinite subtraction games, including the examples above, and provides a complete description of the SG values of all positions:

\begin{proposition}\label{main_prop}
    Let
    \[
        S^{>10} \subseteq (3\mathbb{N}+1)^{>10}
    \]
    be arbitrary.
    Then for the subtraction game with infinite subtraction set
    \[
        S=\{1, 4, 5, 10\} \cup S^{>10},
    \]
    the sets of positions with each SG value are
    \begin{align*}
        \mathcal{H}_0 &= \{0, 2, 8\} \cup \mathcal{H}_0{}^{>10}; \\
        \mathcal{H}_1 &= \{1, 3, 9\} \cup \mathcal{H}_1{}^{>10}; \\
        \mathcal{H}_2 &= \{4, 6\} \cup \mathcal{H}_2{}^{>10}; \\
        \mathcal{H}_3 &= \{5, 7\} \cup \mathcal{H}_3{}^{>10}; \\
        \mathcal{H}_4 &= \{10\}; \\
        \mathcal{H}_g &= \varnothing, \quad \text{for } g\geq 5,
    \end{align*}
    where
    \begin{align*}
        \mathcal{H}_0{}^{>10} &= (3\mathbb{N}+2)^{>10} = \{11, 14, 17, 20,\ldots\}; \\
        \mathcal{H}_1{}^{>10} &= (3\mathbb{N})^{>10} = \{12, 15, 18, 21,\ldots\}; \\
        \mathcal{H}_2{}^{>10} &= (3\mathbb{N}+1)^{>10} \setminus (S^{\geq 10}+6); \\
        \mathcal{H}_3{}^{>10} &= (S^{\geq 10}+6).
    \end{align*}
\end{proposition}

\begin{proof}
    By definition, to prove that a position $u\in\mathbb{N}$ has SG value $g$, it suffices to show:
    (1) for every smaller SG value $g'<g$, $u$ has at least one successor $u'$ with SG value $g'$; and
    (2) $u$ has no successor $u'$ whose SG value is still $g$.

    The cases $u\leq 10$ can be verified by direct computation, so it remains to prove the cases $g<4$ and $u>10$.

    (1) It suffices to indicate how to choose the corresponding $s$ for each $g'<g<4$.

    If $g'=0$, take
    \[
        s=
        \begin{cases}
            1, & g=1; \\
            5, & g=2,3,
        \end{cases}
    \]
    which ensures that $u-s\in \mathcal{H}_0$.
    If $g'=1$, for $g=2,3$ we simply take $s=1$.

    For the remaining case $g'=2, g=3$, we have $u\in S^{\geq 10}+6$, so taking $u'=6\in\mathcal{H}_2$ gives $s=u-u'\in S$.

    (2) That is, for every $g<4$ there are no $u,u'\in\mathcal{H}_g$ such that $u-u'\in S$.

    For $u\in\mathcal{H}_g{}^{>10}$, note that for $s\in S$,
    \[
        s\equiv
        \begin{cases}
            2, & s=5; \\
            1, & s\neq 5
        \end{cases}
        \pmod 3,
    \]
    which is never $0$; but the elements of each $\mathcal{H}_g{}^{>10}$ for $g<4$ are all congruent to a fixed residue modulo $3$, namely
    \[
        g \equiv m \coloneqq
        \begin{cases}
            2, & g=0; \\
            0, & g=1; \\
            1, & g=2,3
        \end{cases}
        \pmod 3.
    \]

    Thus $u'$ can only be one of the elements $u\in\mathcal{H}_g{}^{\leq 10}$ whose residue modulo $3$ is not $m$.
    For each $g<4$, this leaves exactly one possibility:
    \[
        u'=
        \begin{cases}
            0, & g=0; \\
            1, & g=1; \\
            6, & g=2; \\
            5, & g=3.
        \end{cases}
    \]
    Then
    \[
        s\equiv u-u' \equiv
        \begin{cases}
            2, & g=0,1,3;\\
            1, & g=2
        \end{cases}
        \pmod 3.
    \]

    Hence for $g=0,1,3$, we must have $s=5$, and the corresponding values $u=u'+s$ are $5,6,10$, none of which lie in the corresponding $\mathcal{H}_g$.

    For $g=2$, we must have $s\neq 5$; then $u=s+6$. 
    The cases $s\geq 10$ are precisely excluded by the expression for $\mathcal{H}_2$, while the remaining cases $s=1,4$ give $u=7,10$, which are not in $\mathcal{H}_2$ either.
\end{proof}

\begin{corollary}
    (1) For every $k\in\mathbb{N}^{\geq 1}$, there exist infinitely many infinite subtraction games whose outcome sequence and nim‐sequence are both periodic, and such that the period of the latter is $k$ times that of the former.

    (2) There exist infinitely many infinite subtraction games whose outcome sequence is periodic but whose nim‐sequence is not periodic.
\end{corollary}

\begin{proof}
    The outcome sequence of the games constructed in Proposition~\ref{main_prop} always has period $3$; we only need to choose $S$ so that the nim‐sequence has the desired periodicity.

    (1) Take any $p\in(3\mathbb{N}+1)^{>10}$, and define $S$ by
    \begin{align*}
        S^{\leq p} &= \{1,4,5,10\}; \\
        S^{> p}    &= (3k\mathbb{N}+1)^{> p}.
    \end{align*}
    Then by Proposition~\ref{main_prop} the period is $3k$; moreover, infinitely many choices of $p$ yield infinitely many distinct games.

    (2) Besides the Example~\ref{S_e1}, it suffices to take any subtraction set $S$ whose indicator sequence is genuinely not ultimately periodic.
    Clearly there are infinitely many (indeed uncountably many) such choices of $S$.
\end{proof}

\bibliographystyle{plain}
\bibliography{references}

\end{document}